\documentclass[11pt]{amsart}
\usepackage{fullpage,amssymb,amsmath,enumerate,cite,url}

\newtheorem{thm}{Theorem}
\newtheorem{proposition}[thm]{Proposition}
\newtheorem{lemma}[thm]{Lemma}
\newtheorem{fact}[thm]{Fact}
\newtheorem{corollary}[thm]{Corollary}
\newtheorem{convention}[thm]{Convention}
\newtheorem{question}{Question}

\theoremstyle{definition}
\newtheorem{definition}[thm]{Definition}
\newtheorem{remark}[thm]{Remark}

\theoremstyle{remark}

\def\setsep{:\;}
\def\en{\mathbb{N}}
\def\er{\mathbb{R}}
\def\qe{\mathbb{Q}}
\def\M{\mathcal M}
\def\S{\mathcal S}
\def\R{\mathcal R}
\def\F{\mathcal F}
\def\G{\mathcal G}
\def\A{\mathcal A}
\def\N{\mathcal N}

\def\dist{\operatorname{dist}}
\def\spn{\operatorname{sp}\,}
\def\closedSpan{\overline{\spn}}
\def\qespan{\operatorname{sp}_\mathbb{Q}}

\def\aco{\operatorname{aco}}
\def\oo{\sqsubset\!\!\!\sqsupset}

\title{Separable determination in Banach spaces}

\author{Marek C\'uth}
\address{Charles University, Faculty of Mathematics and Physics, Department of Mathematical Analysis, Sokolovsk\'a 83, 186 75 Prague 8, Czech Republic}

\email{cuth@karlin.mff.cuni.cz}

\thanks{The author is a junior researcher in the University Centre for Mathematical Modelling, Applied Analysis and Computational Mathematics (MathMAC). Our investigation was supported by the Research grant GA\v{C}R P201/12/0290.}


\begin{document}

\begin{abstract}
We study a relation between three different formulations of theorems on separable determination  - one using the concept of rich families, second via the concept of suitable models and third, a new one, suggested in this paper, using the notion of $\omega$-monotone mappings. In particular, we show that in Banach spaces all those formulations are in a sense equivalent and we give a positive answer to two questions of O. Kalenda and the author. Our results enable us to obtain new statements concerning separable determination of $\sigma$-porosity (and of similar notions) in the language of rich families; thus, not using any terminology from logic or set theory.

Moreover, we prove that in Asplund spaces, generalized lushness is separably determined.
\end{abstract}
\maketitle

Let $X$ be a nonseparable topological space and let $\S(X)$ denote the family of all closed separable subspaces of $X$. Below we consider both topological spaces and also the special case of Banach spaces where by a ``subspace'' we mean a linear subset. %For the study of nonseparable Banach spaces, construction of a separable subspace with certain property is sometimes important. It enables us to transfer properties from smaller (separable) spaces to larger ones.
One of the important methods of proofs in the nonseparable theory is the ``separable reduction''. By a separable reduction we usually mean the possibility to extend the validity of a statement from separable spaces to the nonseparable setting without knowing the proof of the statement in the separable case. This method has been used in the setting of Banach spaces e.g. in \cite{c, cuthFab, cr, crz, FaIo, LiPrTi, MoSp, Pr84, vz, z} and in the general setting of topological spaces e.g. in \cite{c, LiMo08}. Experience shows that an optimal method of separable reduction is to prove that certain notions are ``separably determined''. Roughly speaking, a statement $\phi$ concerning a topological space $X$ is considered to be \emph{separably determined} if there exists a sufficiently large family $\F\subset \S(X)$ such that for every $F\in\F$ we have
\begin{equation*}\label{eq:sepDet}\tag{$\maltese$}
\text{The statement }\phi\text{ holds in $X$ }\Longleftrightarrow \text{Analogous statement }\phi_F\text{ holds in $F$}.
\end{equation*}

As far as the author knows, this approach started in \cite{Pr84} and the notion of a separable determination was for the first time explicitly mentioned in \cite{LiPrTi}. Let us illustrate the use of separable determination by recalling the strategy of the proof of the result from \cite{Pr84}, which says that every Gateaux differentiable Lipschitz function $f:X\to \er$ on an Asplund space is Fr\'echet differentiable at some point. It consists of two steps:
\begin{itemize}
	\item[1.]{\it (proof for separable spaces)} The result holds if $X$ is Asplund and separable.
	\item[2.]{\it (separable determination)} If $X$ is Asplund and not separable, there is $F\in\S(X)$  such that
		\begin{equation}\label{eq:1}\forall x\in F:\quad f|_F\text{ is Fr\'echet differentiable at }x\implies f\text{ is Fr\'echet differentiable at }x.\end{equation}		
\end{itemize}
The statement easily follows. Indeed, let $X$ be an Apslund space which is not separable and let $f:X\to\er$ be Gateaux differentiable and Lipschitz. Construct the separable subspace $F$ such that \eqref{eq:1} holds. By the assumptions and since the statement holds for separable spaces, %Since $f|_F$ is Gateaux differentiable and Lipschitz and since $F$ is Asplund and separable, 
there is $x\in F$ such that $f|_F$ is Fr\'echet differentiable at $x$. Therefore, by \eqref{eq:1}, $f$ is Fr\'echet differentiable at $x$.

As demonstrated above, one makes the final deduction using just one separable subspace, however in order to combine finitely many results together, it is convenient to know that the family $\F$ is ``sufficiently large''. Let us make it clear with an example. Let $X$ be a Banach space, $A\subset X$ a Borel set and $f:X\to\er$ a function. By \cite[Proposition 4.1 and Theorem 5.10]{c}, there are ``sufficiently large'' families $\F_1,\F_2\subset\S(X)$ such that for every $F\in\F_1$ we have
\begin{equation}\label{eq:2}\forall x\in F:\quad f|_F\text{ is Fr\'echet differentiable at }x\iff f\text{ is Fr\'echet differentiable at }x,\end{equation}
and for every $F\in\F_2$ we have
\begin{equation}\label{eq:3}A\cap F\text{ is dense in }F\iff A\text{ is dense in }X.\end{equation}
Hence, if we consider $A = \{x\in X\setsep f\text{ is Fr\'echet differentiable at }x\}$ and pick $F\in\F_1\cap\F_2$, by \eqref{eq:2} we have
$A\cap F = \{x\in F\setsep f|_F\text{ is Fr\'echet differentiable at }x\}$ and so by \eqref{eq:3} we get
\[f\text{ is Fr\'echet differentiable on a dense set}\iff f|_F\text{ is Fr\'echet differentiable on a dense set}.\]
Let us emphasize that in order to pick $F\in\F_1\cap\F_2$, we need to know that $\F_1\cap\F_2\neq\emptyset$. Therefore, it is convenient to know that the intersection of finitely many ``sufficiently large'' families is not empty. It seems that the right approach is through one of the following concepts.

One of them is the \emph{rich family} introduced in \cite{boMo} by J. M. Borwein, W. Moors and then used  in the Banach space theory \cite{cuthFab, FaIo, LiPrTi, MoSp, vz, z} and in the setting of topological spaces \cite{LiMo08}.

\begin{definition}Let $X$ be a topological space. A family $\F\subset \S(X)$ is called {\em rich} if
\begin{enumerate}[\upshape (i)]
	\item $\F$ is \emph{cofinal}, that is, each separable subspace of $X$ is contained in an element of $\F$, and
	\item $\F$ is \emph{$\sigma$-closed}, that is, for every increasing sequence $F_i$ in $\F$, $\overline{\bigcup_{i=1}^\infty F_i}$ belongs to $\F$.
\end{enumerate}
\end{definition}

\noindent The nonemptyness of the intersection of finitely many families is then witnessed by the following.

\begin{proposition}[{\cite[Proposition 3.1]{LiMo08}}]\label{p:boMO}
Suppose that $X$ is a topological space. If $\{\F_n\setsep n\in\en\}$ are rich families then so is $\bigcap_{n\in\en}\F_n$.
\end{proposition}

Another concept arises from the ``method of suitable models'' (sometimes called also ``method of elementary submodels'') used e.g. in \cite{c, cr, crz}.

\begin{definition}Let $X$ be a topological space. We say $\F\subset S(X)$ is \emph{large in the sense of suitable models} if there exists a finite list of formulas $\Phi$ and a countable set $Y$ such that \[\F = \{\overline{X\cap M}\setsep M\text{ is a suitable model for $\Phi$ containing $Y$}\}.\]
\end{definition}
 
\noindent We refer to the next section where more details about the concept of suitable models may be found.

Our first main result is that when dealing with a Banach space, those two concepts are equivalent when dealing with separable determination, that is, the existence of a rich family $\F$ satisfying $(\maltese)$ is equivalent to the existence of a family $\F'$ large in the sense of suitable models satisfying $(\maltese)$.

\begin{thm}\label{t:main1}Let $X$ be a Banach space and let us have $\F\subset\S(X)$.
	\begin{itemize}
		\item[(i)] If $\F$ is rich, then there exists $\F'\subset \F$ which is large in the sense of suitable models.
		\item[(ii)] If $\F$ is large in the sense of suitable models, then there exists $\F'\subset \F$ which is rich.
	\end{itemize}
Moreover the statement (i) holds also if $X$ is an arbitrary topological space.
\end{thm}

The statement (i) follows actually from the proof of \cite[Proposition 3.1]{cuthKalenda}. The statement (ii) is proved in Secton \ref{section:richVsSuitable}, actually it was almost proved in \cite{cuthKalenda}; however, the authors did not realize it and this lead them to formulate three questions, two of which we answer here, see Section \ref{section:richVsSuitable}. Even though the proof of (ii) itself is quite short, the implication seems to be important and it can be considered as the main result of the whole paper. Let us comment it. When dealing with simple statements, the proof through both concepts (rich families and suitable models) is more or less of the same difficulty. However, the strength of the concept of suitable models reveals when dealing with complicated proofs where inductive construction (used when constructing separable space using standard methods) would be tough. This is the case e.g. of $\sigma$-porosity, where separable determination was not known in any sense for quite a long time (e.g. in the book of J. Lindenstrauss, D. Preiss an J. Ti\v{s}er the authors proved one implication using the concept of rich families \cite[Corollary 3.6.7]{LiPrTi}, but not the equivalence) and it has been recently obtained using the concept of suitable models \cite{cr}, \cite{crz}. The concept of suitable models seems to be a powerful tool. However, the formulation of corresponding separable determination theorems involves notions from set theory and logic, which makes the results less applicable by non experts. Therefore, it is important to know that from a separable determination statement proved using the concept of suitable models a statement formulated in the language of rich families follows. This is exactly what Theorem \ref{t:main1} says. For example, it enables us to reformulate the above mentioned results concerning e.g. $\sigma$-porosity in the language of rich families. Some of those results have already been applied (see e.g. \cite[Proposition~CR]{vz}) and since ``$\sigma$-porosities'' naturally occur in the theory of separable Banach spaces, we hope that the result below will enable more authors to prove their theorems in a non-separable setting using separable reduction. We apply (ii) in Theorem~\ref{t:main1} to \cite[Theorem~5.10]{crz}, \cite[Theorems~5.1 and 5.4]{cr} and \cite[Theorem~4.7]{c}. Recall that every Borel set is Suslin; hence, our result applies for example to Borel sets.

\begin{corollary}\label{c:application}
Let $X$ be a Banach space and $A \subset X$ be a Souslin set.
Then there exists a rich family $\F\subset\S(X)$ such that for every $V\in\F$ we have
\begin{align*}
A \text{ is meager in the space } X           &\iff A \cap V \text{ is meager in the subspace } V,\footnotemark[1]\\
A \text{ is $\sigma$-upper porous in the space } X           &\iff A \cap V \text{ is $\sigma$-upper porous in the subspace } V, \\
A \text{ is $\sigma$-lower porous in the space } X           &\iff A \cap V \text{ is $\sigma$-lower porous in the subspace } V,
\end{align*}
Moreover, if $X$ is Asplund and $\alpha\in [0,1)$, there exists a rich family $\F\subset\S(X)$ such that for every $V\in\F$ we have
\begin{align*}
A \text{ is $\sigma$-$\alpha$-cone porous in the space } X  &\iff A \cap V \text{ is $\sigma$-$\alpha$-cone porous in the subspace } V,\footnotemark[2] \\
A \text{ is cone small in the space } X                     &\iff A \cap V \text{ is cone small in the subspace } V.
\end{align*}
\end{corollary}\footnotetext[1]{It is known to the author that it is not to difficult to prove the statement about meagerness using classical methods (one implication follows from \cite[Lemma 4.6]{z}, see also \cite[page 44]{LiPrTi}); however, we did not find any reference for that (even though we believe this is a common knowledge for people working with rich families). Therefore, we believe it is of some importance to mention it here.}
\footnotetext[2]{In \cite[Theorem 5.10]{crz} the statement concerning $\sigma$-$\alpha$-cone porosity is formulated only for rational $\alpha$; however, the same proof gives us the result for $\alpha\in M$ and, since the family is constructed after $\alpha$ is chosen, we may assume that we have $\alpha\in M$}

We suggest one more approach to theorems on separable determination using the notion of $\omega$-monotone mappings, a concept which has been considered already in topology, see e.g. \cite{rojTka}.

\begin{definition}\label{def:omega}Given two infinite sets $I$ and $J$, a function $\phi:[I]^{\leq\omega}\to [J]^{\leq\omega}$ is called \emph{$\omega$-monotone} provided that:
	\begin{itemize}
  	\item[(i)] $\phi$ is \emph{monotone}, that is, if $A\subset B$ are countable subsets of $I$, then $\phi(A)\subset\phi(B)$;
  	\item[(ii)] if $(A_n)$ is an increasing sequence of countable subsets of $I$, then $\phi(\bigcup_{n=1}^\infty A_n)=\bigcup_{n=1}^\infty \phi(A_n)$.
 \end{itemize}
\end{definition}

\begin{definition}Let $X$ be a topological space. We say $\F\subset S(X)$ is \emph{large in the sense of $\omega$-monotone mappings} if there exists an $\omega$-monotone mapping $\phi:[X]^{\leq\omega}\to [X]^{\leq\omega}$ such that $\phi(C)\supset C$ for every $C\in[X]^{\leq\omega}$ and \[\F = \{\overline{\phi(C)}\setsep C\in[X]^{\leq\omega}\}.\] We say $\phi$ is a \emph{witnessing map} for $\F$.
\end{definition}

\noindent It is quite easy to see that the intersection of finitely many families large in the sense of $\omega$-monotone mappings is nonempty, see Proposition \ref{p:combineMon}. The advantage of this new approach to separable determination statements when compared with ``suitable models'' and ``rich families'' is the following. First, it does not require any knowledge of set theory or logic. Next, when compared with the concept of rich families, one does not have to check the $\sigma$-closeness, which in some cases might cause difficulties --- as an example one might have a look at the proof of Proposition \ref{p:redukceGL} where it seems to us to be unclear whether the family $\F$ constructed in the proof is rich. 

Our second main result is that when dealing with a Banach space, this concept is in a sense equivalent to the previous ones.

\begin{thm}\label{t:main2}Let $X$ be a Banach space and let us have $\F\subset\S(X)$.
	\begin{itemize}
		\item[(i)] If $\F$ is rich, then there exists $\F'\subset \F$ which is large in the sense of $\omega$-monotone mappings.
		\item[(ii)] If $\F$ is large in the sense of $\omega$-monotone mappings, then there exists $\F'\subset \F$ which is rich.
	\end{itemize}
\end{thm}

Finally, in Section~\ref{section:GL} we prove a new separable determination result.

 \begin{thm}\label{t:redukceGL}Let $X$ be a Banach space. Then there exists a rich family $\F\subset\S(X)$ such that for every $V\in\F$ we have
 	\[
		X\text{ is generalized lush }\Longrightarrow V\text{ is generalized lush}.
	\]
 Moreover, if $X$ is an Asplund space, there exists a rich family $\F\subset\S(X)$ such that for every $V\in\F$ we have
	\[
		X\text{ is generalized lush }\iff V\text{ is generalized lush}.
	\]
\end{thm}

This result could be also considered as a raison d'etre for Theorem \ref{t:main2}, because we do not know of a direct argument which would give us a rich family and so in our argument we consider also a family which is large in the sense of $\omega$-monotone mappings, see Proposition \ref{p:redukceGL}. Using Theorem~\ref{t:redukceGL}, we reprove the result that every Asplund lush space is generalized lush. We refer to Section~\ref{section:GL} for the definition and some more details concerning the (generalized) lushness. At the 43\textsuperscript{rd} Winter School of Abstract Analysis, J.-D. Hardtke asked whether every nonseparable lush Banach space is generalized lush. Up to our knowledge, this question is still open.

\begin{question}\label{q1}
Is every lush Banach space generalized lush?
\end{question}

\begin{question}\label{q2}
Is it true that generalized lushness is separably determined (i.e. do we have an equivalence in Theorem~\ref{t:redukceGL} not assuming $X$ is Asplund)?
\end{question}

Note that positive answer to Question~\ref{q2} would imply positive answer to Question~\ref{q1}.\\[3pt]

Let us fix some notations. The set of rational numbers is denoted by $\qe$. For an infinite set $M$ the symbol $[M]^{\,\le\omega}$ means the family of all at most countable subsets of $M$. All linear spaces are over the field $\er$. Let $(X,\|\cdot\|)$ be a Banach space. We denote by $X^*$ its dual, by $B_X$ its closed unit ball and by $S_X$ its unit sphere. For a set $A\subset X$ the symbols $\aco A$, $\spn A$ and $\closedSpan A$, and $\qespan A$ mean the absolutely convex hull of $A$, the linear span of $A$, the norm-closed linear span of $A$ and the set consisting of all finite linear combinations of elements in $A$ with rational coefficients, respectively.

\section{Comparision of the concept of rich families with the concept of suitable models}\label{section:richVsSuitable}

In this section we compare the concept of rich families with the concept of suitable models. Our main result is that in Banach spaces both concepts are equivalent, see Theorem \ref{t:main1}. Actually, we prove even something more, namely that suitable models generate nice rich families in any Banach space, see Theorem \ref{t:gener}. This gives a positive answer to \cite[Question 2.8]{cuthKalenda} and implies a positive answer to \cite[Question 3.6]{cuthKalenda}.

The language of suitable models allows the central ideas to emerge from what would otherwise be a mass of technical details. The method of replacing an inductive construction by ``suitable model'' was used in topology already in 1988 by A. Dow \cite{dow}, in the Banach space theory in 2005 by P. Koszmider \cite{kos}, later in 2009 by W. Kubi{\'s} \cite{kubis} and even later in 2012 by the author \cite{c} who simplified its presentation to the form which we will use here.

Let us recall some basics concerning the method of suitable models. A brief description of it can be found in \cite{crz}; for a more detailed one, see \cite{c}. Let $N$ be a fixed set and $\phi$ a formula in the basic language of the set theory. By the \emph{relativization of $\phi$ to $N$} we understand the formula $\phi^N$ which is obtained from $\phi$ by replacing each chain of the form ``$\forall x$'' by ``$\forall x \in N$'' and each chain of the form ``$\exists x$'' by ``$\exists x \in N$''. Let $\phi(x_1,\ldots,x_n)$ be a formula with all free variables shown, that is, a formula whose free variables are exactly $x_1,\ldots,x_n$. We say \emph{$\phi$ is absolute for $N$} if
\[
\forall a_1, \ldots, a_n \in N\colon \bigl(\phi^N(a_1,\ldots,a_n) \leftrightarrow \phi(a_1,\ldots,a_n)\bigr).
\]
The method is based mainly on the following theorem (a proof can be found in \cite[Chapter IV, Theorem 7.8]{k}).
The cardinality of a set $A$ is denoted by $|A|$.

\begin{thm}\label{T:countable-model}
Let $\phi_1, \ldots, \phi_n$ be any formulas and $X$ be any set. Then there exists a set $M \supset X$ such that
$\phi_1, \ldots, \phi_n \text{ are absolute for } M$ and $|M| = \max\,(\aleph_0,|X|)$.
\end{thm}

The following notation is useful.

\begin{definition}
\rm Let $\Phi$ be a finite list of formulas and $X$ be any countable set.
Let $M \supset X$ be a {\tt countable} set such that each $\phi$ from $\Phi$ is absolute for $M$.
Then we say that $M$ \emph{is a suitable model for $\Phi$ containing $X$}.
This is denoted by $M \prec (\Phi; X)$.
\end{definition}

\noindent Let us emphasize that a suitable model in our terminology is always countable.

We shall also use the following convention.

\begin{convention}\hfil\\
%\begin{itemize}
$\bullet$ If $(X, +, \cdot, \| \cdot \|)$ is a normed linear space and $M$ is a suitable model (for some finite list of formulas containing some set), then by writing $X\in M$ (or by writing $\{X\}\subset M$) we mean that 
$X,\;+,\;\cdot,\;\text{\mbox{$\| \cdot \|$}} \in M$.\\
$\bullet$ If $X$ is a topological space and $M$ is a suitable model, then we denote by $X_M$ the set $\overline{X\cap M}$; clearly, the set $X_M$ is separable.
%\end{itemize}
\end{convention}

Finally, we recall several results about suitable models
(the proofs are easy and they can be found in \cite[Sections 2 and 3]{c}).

\begin{lemma}\label{l:basics-in-M}
There are a finite list of formulas $\Phi$ and a countable set $C$ such that any $M \prec (\Phi; C)$ satisfies the following:
\begin{enumerate}[\upshape (i)]
	\item Let $f$ be a function such that $f\in M$. Then, for every $x \in \operatorname{Dom} f \cap M$ we have $f(x) \in M$.
	\item If $X$ is a normed linear space and $X\in M$, then $X_M := \overline{X\cap M}$ is a (closed separable) linear subspace.
\end{enumerate}
\end{lemma}

Now, let us come to our observation from which all the results in this section follow. We need one more notion which comes from \cite{cuthKalenda}.

\begin{definition}\label{d:richNice}\rm Let $X$ be a topological space. We say that \emph{suitable models generate nice rich families in $X$}, if the following holds:\\[3pt]
Whenever $Y$ is a countable set and $\Phi$ is a finite list of formulas, there exists a family $\M$ of sets satisfying the following conditions:
\begin{enumerate}[\upshape (i)]
	\item $\forall M\in\M:\quad M\prec(\Phi;Y)$.
	\item The set $\{X_M:\; M\in\M\}$ is a rich family of separable subspaces in $X$.
	\item $\forall M,N\in\M:\quad M\subset N \Longleftrightarrow X_M \subset X_N$.
\end{enumerate}
\end{definition}

\begin{thm}\label{t:gener}Let $X$ be a topological space which is homeomorphic to a Banach space. Then suitable models generate nice rich families in $X$.\end{thm}
\begin{proof}We may without loss of generality assume $X$ is not separable. By the result of H. Toru\'nczyk \cite{To81}, all the infinite-dimensional Banach spaces with the same density are topologically homeomorphic. Hence, if $H$ is a Hilbert space of the same density as $X$, we may pick $f:X\to H$ a homeomorphism onto (not necessarily linear). In order to see that suitable models generate nice rich families in $X$, fix a countable set $Y$ and a finite list of formulas $\Phi$. We may without loss of generality assume that  $X$, $f$, $f^{-1}$, $H$ and all the sets from the statement of Lemma~\ref{l:basics-in-M} are elements of $Y$ and that our $\Phi$ contains all the formulas from the statement of Lemma~\ref{l:basics-in-M} (because the condition (i) in Definition~\ref{d:richNice} is inherited by countable subsets and shorter sublists of formulas).

By \cite[Theorem 2.7]{cuthKalenda}, suitable models generate nice rich families in $H$; hence, we may pick a family $\M$ satisfying (i)--(iii) in Definition~\ref{d:richNice} for the space $H$. We will show that this family is sufficient for the space $X$ as well. First, notice that $f(X_M) = H_M (=\overline{H\cap M})$. Indeed, by Lemma~\ref{l:basics-in-M}, we have $f(X\cap M)\subset H\cap M$ and $f^{-1}(H\cap M)\subset X\cap M$; hence, $f(X\cap M) = H\cap M$ and since $f$ is homeomorphism, $f(X_M) = H_M$. Thus, whenever $M$, $N\in \M$, we have
\[
M\subset N \Longleftrightarrow \overline{H\cap M}\subset \overline{H\cap N}\Longleftrightarrow \overline{X\cap M}\subset \overline{X\cap N}.
\]
It remains to show that $\{X_M:\; M\in\M\}$ is a rich family of separable subspaces in $X$. This easily follows from the above and the fact that, by the choice of $\M$, $\{H_M:\; M\in\M\}$ is a rich family of separable subspaces in $H$.
\end{proof}

\begin{proof}[Proof of Theorem \ref{t:main1}]Statement (i) follows from \cite[Proposition 3.1]{cuthKalenda}, statement (ii) from Theorem \ref{t:gener}. The ``moreover'' part follows from the fact that the proof of \cite[Proposition 3.1]{cuthKalenda} works also for topological spaces (not only for Banach spaces for which it is formulated).
\end{proof}

\begin{remark}Let us emphasize that statements from this Section are metamathematical rather than mathematical (because  existence of certain formulas is required in their statements). Therefore, one should realize we are using methods of mathematical logic in order to show the existence of certain formulas in our proofs.
\end{remark}

\section{Comparision of the concept of rich families with the concept of $\omega$-monotone mappings}\label{section:richVsMonotone}

In this section we give some basic observations concerning the concept of $\omega$-monotone mappings and we compare it with the concept of rich families. Our main result is that in Banach spaces both concepts are equivalent, see Theorems \ref{t:rich->mon} and \ref{t:mon->rich}. We start with an observation which enables us to combine finitely many results together.

\begin{proposition}\label{p:combineMon}Let $X$ be a topological space and let $\F_1,\ldots,\F_n\subset \S(X)$ be a finite number of families large in the sense of $\omega$-monotone mappings. Then there exists $\F\subset\bigcap_{i=1}^n\F_i$ which is large in the sense of $\omega$-monotone mappings.
\end{proposition}
\begin{proof}It suffices to give the proof for the intersection of two families and then use induction. Let $\F_1$, $\F_2$ be two families large in the sense of $\omega$-monotone mappings and let $\phi_1$, $\phi_2$ be the corresponding witnessing maps. Let us define an $\omega$-monotone mapping $\phi:[X]^{\leq\omega}\to [X]^{\leq\omega}$ by $\phi(C):=\bigcup_{n\in\en}(\phi_2\phi_1)^n(C)$, $C\in[X]^{\leq\omega}$. Then
\[
\phi(C)\subset \phi_1(\phi(C)) = \bigcup_{n\in\en}\phi_1(\phi_2\phi_1)^n(C)\subset \bigcup_{n\in\en}(\phi_2\phi_1)^{n+1}(C) = \phi(C).
\]
Thus, we have $\phi(C) = \phi_1(\phi(C))$. Since it follows that
\[
\phi_1(\phi(C))\subset \phi_2(\phi_1(\phi(C))) = \bigcup_{n\in\en}(\phi_2\phi_1)^{n+1}(C) = \phi(C) = \phi_1(\phi(C)),
\]
we have $\phi(C) = \phi_1(\phi(C)) = \phi_2(\phi_1(\phi(C)))$. Therefore, it suffices to put $\F:=\{\overline{\phi(C)}\setsep C\in[X]^{\leq\omega}\}$.
\end{proof}

\begin{fact}\label{fact}Let $X$ be a topological vector space. Then the family of all the closed separable linear subspaces of $X$ is large in the sense of $\omega$-monotone mappings.
\end{fact}
\begin{proof}The witnessing map is given by $[X]^{\leq\omega}\ni C\mapsto \qespan C\in[X]^{\leq\omega}$.
\end{proof}

\begin{thm}\label{t:rich->mon}Let $X$ be a topological space and let $\F\subset \S(X)$ be a rich family. Then there exists $\F'\subset \F$ which is large in the sense of $\omega$-monotone mappings.
\end{thm}
\begin{proof}For every $x\in X$, let us pick one $F_x\in\F$ with $x\in F_x$ and for every $F\in\F$, let us pick a countable set $D_F\subset F$ which is dense in $F$. Moreover, for every $G,H\in\F$, pick $F_{G,H}\in\F$ with $F_{G,H}\supset G\cup H$. Now, let us define a function $\G:[X]^{\leq\omega}\to [\F]^{\leq\omega}$ by putting for every $C\in[X]^{\leq\omega}$
\[\begin{split}
	\G_1(C) & : = \{F_x\setsep x\in C\},\\
	\G_{n+1}(C) & := \G_n(C)\cup \{F_{G,H}\setsep G,H\in\G_n(C)\},\qquad n\in\en,\\
	\G(C) & : = \bigcup_{n\in\en}\G_n(C).
\end{split}\]
It follows from the construction that $\G(C)$ is up-directed. Moreover, $\G$ is an $\omega$-monotone mapping. Indeed, it is easy to see that $\G$ is monotone. In order to prove $\omega$-monotonicity, pick an increasing sequence $(C_k)_{k\in\en}$ of countable subsets of $X$. By monotonicity, we have $\G(\bigcup_{k=1}^\infty C_k)\supset \bigcup_{k=1}^\infty \G(C_k)$. By induction on $n\in\en$, it is straightforward to verify that $\G_n(\bigcup_{k\in\en} C_k)\subset \bigcup_{k=1}^\infty \G(C_k)$; hence, we have $\G(\bigcup_{k=1}^\infty C_k)\subset \bigcup_{k=1}^\infty \G(C_k)$. This proves that $\G$ is an $\omega$-monotone mapping.

Finally, define a function $\phi:[X]^{\leq\omega}\to [X]^{\leq\omega}$ by $\phi(C):=C\cup \bigcup\{D_F\setsep F\in\G(C)\}$, $C\in[X]^{\leq\omega}$. Since $\G$ is $\omega$-monotone, $\phi$ is also $\omega$-monotone. Fix $C\in[X]^{\leq\omega}$. By the definition of $\phi$, we have $\phi(C)\supset C$ and $\overline{\phi(C)} = \overline{\bigcup\{F\setsep F\in\G(C)\}}$. Since $\G(C)$ is a countable and up-directed set, there is an increasing sequence $(G_n)_{n\in\en}$ consisting of elements from $\G(C)$ with $\overline{\bigcup\{F\setsep F\in\G(C)\}} = \overline{\bigcup G_n}$.  Hence, $\overline{\phi(C)} = \overline{\bigcup G_n}$ and since $\F$ is $\sigma$-closed, we have $\overline{\phi(C)}\in\F$. Now, it is enough to put $\F':=\{\overline{\phi(C)}\setsep C\in[X]^{\leq\omega}\}$.
\end{proof}

Now, we will proof a converse to Theorem \ref{t:rich->mon} for topological spaces homeomorphic to a Banach space. The key observation is the following.

\begin{lemma}\label{l:podzobrazeni}Let $X$ be a topological space, $\phi:[X]^{\leq\omega}\to [X]^{\leq\omega}$ an $\omega$-monotone mapping such that $\phi(C)\supset C$ for every $C\in[X]^{\leq\omega}$ and let $D\in [X]^{\leq\omega}$ be such that $\phi(\{x_1,\ldots,x_n\})\subset \overline{D}$ for every $x_1,\ldots,x_n\in D$. Then $\overline{D} = \overline{\phi(D)}$.
\end{lemma}
\begin{proof}First, let us observe that we have
\[
\overline{D} = \overline{\bigcup\{\phi(\{x_1,\ldots,x_n\})\setsep x_1,\ldots,x_n\in D\}}.
\]
Indeed, the inclusion ``$\supset$'' follows from the assumptions and the inclusion ``$\subset$'' follows from the fact that $x\in\phi(\{x\})$ for every $x\in D$.
Let us denote the points of $D$ by $\{d_n\}_{n=1}^\infty$. By the $\omega$-monotonicity of $\phi$ we have $\bigcup_{n\in\en}\phi(\{d_1,\ldots,d_n\}) = \phi(D)$. Moreover, by the monotonicity of $\phi$ we have $\bigcup\{\phi(\{x_1,\ldots,x_n\})\setsep x_1,\ldots,x_n\in D\} = \bigcup_{n\in\en}\phi(\{d_1,\ldots,d_n\})$; hence, $\overline{\phi(D)} = \overline{\bigcup\{\phi(\{x_1,\ldots,x_n\})\setsep x_1,\ldots,x_n\in D\}}$ which finishes the proof.
\end{proof}

It follows from the Lemma above that having an $\omega$-monotone mapping $\phi$ as above, we have that $X_M = \overline{\phi(X\cap M)}$ for every suitable model; hence, there is a large family $\F'$ in the sense of suitable models with $\F'\subset \{\overline{\phi(C)}\setsep C\in[X]^{\leq\omega}\}$. By Theorem \ref{t:gener}, in topological spaces homeomorphic to a Banach space we get the needed rich family. In order to make the argument more transparent for non-experts in set theory, we present here a proof not involving suitable models.

\begin{thm}\label{t:mon->rich}Let $X$ be a topological space which is homeomorphic to a Banach space. If $\F\subset \S(X)$ is large in the sense of $\omega$-monotone mappings, then there exists a rich family $\F'\subset \S(X)$ with $\F'\subset \F$.
\end{thm}
\begin{proof}By \cite{To81}, all the Banach space with the same density are topologically homeomorphic. Hence, $X$ is topologically homeomorphic to a Hilbert space $H$. First, let us observe that it suffices to consider the case when $X = H$. Indeed, suppose we know the theorem holds for a Hilbert space $H$ and let $f:X\to H$ be a homeomorphism onto. Then if $\F\subset\S(X)$ is  large in the sense of $\omega$-monotone mappings with a witnessing map $\phi$, the map $\psi:=f\circ\phi\circ f^{-1}:[H]^{\leq\omega}\to [H]^{\leq\omega}$ is a witness of the fact that $\{\overline{\psi(C)}\setsep C\in[X]^{\leq\omega}\}\subset \S(H)$ is large in the sense of $\omega$-monotone mappings and therefore there exists $\N\subset [X]^{\leq\omega}$ such that $\{\overline{\psi(C)}\setsep C\in\N\}$ is rich. Now, it is straightforward to show that $\{\overline{\phi(C)}\setsep C\in\N'\}\subset\F$ is rich, where $\N' := \{f^{-1}(C)\setsep C\in\N\}$.

Thus, let us assume that $X$ is a Hilbert space. Let $I$ be a subset of $X$ such that $\closedSpan I = X$ and $i\notin\closedSpan(I\setminus\{i\})$ for every $i\in I$ (one may take e.g. an orthonormal basis of $X$). For each $x\in X$, pick $A_x\in[I]^{\leq\omega}$ with $x\in \closedSpan A_x$ and define an $\omega$-monotone mapping $\psi:[X]^{\leq\omega}\to [I]^{\leq\omega}$ by $\psi(C):=\bigcup\{A_x\setsep x\in C\}$, $C\in[X]^{\leq\omega}$.  Let $\phi$ be a witnessing map for $\F$. Finally, define a function $\A:[I]^{\leq\omega}\to [X]^{\leq\omega}$ by putting for every $A\in[I]^{\leq\omega}$
\[\begin{split}
	\A_1(A) & : = A,\\
	\A_{n+1}(C) & := \qespan\Bigg(\A_n(A)\cup \bigcup\bigg\{\psi\Big(\phi\big(\{x_1\ldots,x_k\}\big)\Big)\setsep x_1,\ldots,x_k\in\A_n(A),\,k\in\en\bigg\}\Bigg),\qquad n\in\en,\\
	\A(A) & : = \bigcup_{n\in\en}\A_n(A).
\end{split}\]
Our aim is to show that $\F':= \{\overline{\A(A)}\setsep A\in[I]^{\leq\omega}\}$ works, that is, $\F'\subset \F$ and it is a rich family. Fix $A\in[I]^{\leq\omega}$. By induction on $n\in\en$, it is straightforward to check that $\phi(\{x_1,\ldots,x_k\})\subset \overline{\A(A)}$ for every $x_1,\ldots,x_k\in\A_n(A)$ and $k\in\en$; hence, by Lemma \ref{l:podzobrazeni}, we have $\overline{\A(A)} = \overline{\phi(\A(A))}$ which proves that $\F'\subset\F$. In order to prove the cofinality of $\F'$, fix a separable subspace $S\subset X$. Since $I$ is linearly dense in $X$, there exists $A\in[I]^{\leq\omega}$ with $S\subset \closedSpan(A)$. But since $A\subset \A(A)$ and $\overline{\A(A)}$ is linear, we have $\closedSpan A\subset \overline{\A(A)}\in\F'$; hence, $S\subset \overline{\A(A)}\in\F'$ which proves the confinality of $\F'$.

Finally, in order to see that $\F'$ is $\sigma$-closed, let us start with observing the following
\begin{itemize}
	\item[(a)] $\A$ is $\omega$-monotone;
	\item[(b)] $\overline{\A(A)} = \closedSpan(\A(A)\cap I)$ for every $A\in[I]^{\leq\omega}$;
	\item[(c)] $\A(\A(A)\cap I) = \A(A)$ for every $A\in[I]^{\leq\omega}$.
\end{itemize}
It is easy to see that $\A$ is monotone. In order to prove $\omega$-monotonicity, pick an increasing sequence $(A_k)_{k\in\en}$ of countable subsets of $I$. By monotonicity, we have $\A(\bigcup_{k=1}^\infty A_k)\supset \bigcup_{k=1}^\infty \A(A_k)$. By induction on $n\in\en$, it is straightforward to verify that $\A_n(\bigcup_{k\in\en} A_k)\subset \bigcup_{k=1}^\infty \A(A_k)$; hence, we have $\A(\bigcup_{k=1}^\infty A_k)\subset \bigcup_{k=1}^\infty \A(A_k)$. This proves (a). In order to prove (b), fix $A\in[I]^{\leq\omega}$. Since $\overline{\A(A)}$ is linear, we have $\overline{\A(A)} \supset \closedSpan(\A(A)\cap I)$. On the other hand, let us prove by induction on $n\in\en$ that $\A_n(A)\subset \closedSpan(\A(A)\cap I)$. This is obvious if $n=1$. Let us assume that $\A_n(A)\subset \closedSpan(\A(A)\cap I)$. Since $\closedSpan(\A(A)\cap I)$ is linear, in order to see that $\A_{n+1}(A)\subset \closedSpan(\A(A)\cap I)$ it suffices to observe that $\psi\Big(\phi\big(\{x_1\ldots,x_k\}\big)\Big)\subset \closedSpan(\A(A)\cap I)$ whenever $x_1,\ldots,x_k\in\A_n(A)$ which follows from the fact that $\psi\Big(\phi\big(\{x_1\ldots,x_k\}\big)\Big)\subset \A_{n+1}(A)\cap I$ for every $x_1,\ldots,x_k\in\A_n(A)$. Therefore, (b) holds. It follows from the construction that we have $\A_n(\A(A)\cap I) = \A(A)$ for $n\geq 2$; hence, (c) holds.

Finally, in order to see that $\F'$ is $\sigma$-closed, let us assume that we have an increasing sequence $\overline{\A(A_n)}$. By (b), $\closedSpan(\A(A_n)\cap I)$ is increasing and so we have $\overline{\bigcup_{n\in\en} \closedSpan(\A(A_n)\cap I)} = \closedSpan\left(\bigcup_{n\in\en} \A(A_n)\cap I\right)$. Since $i\notin\closedSpan(I\setminus\{i\})$ for every $i\in I$, we have
\[
	\closedSpan A\subset \closedSpan B\implies A\subset B,\qquad A,B\in [I]^{\leq\omega}.
\]
Hence, since $\closedSpan(\A(A_n)\cap I)$ is increasing, $\A(A_n)\cap I$ is increasing as well. Putting everything together we have
\[\begin{split}
	\overline{\bigcup \overline{A_n}} \, & \stackrel{(b)}{=} \, \overline{\bigcup_{n\in\en} \closedSpan(\A(A_n)\cap I)} = \closedSpan\left(\bigcup_{n\in\en} \A(A_n)\cap I\right) \stackrel{(c)}{=}\closedSpan\left(\bigcup_{n\in\en} \A(\A(A_n)\cap I)\cap I\right)  \stackrel{(a)}{=}\\
	& =  \closedSpan\left(\A\left(\bigcup_{n\in\en} \A(A_n)\cap I\right)\cap I\right) \stackrel{(b)}{=} \overline{\A\left(\bigcup_{n\in\en} \A(A_n)\cap I\right)}\in\F'
\end{split}\]
and $\F'$ is $\sigma$-closed.
\end{proof}

\section{Separable determination of (generalized) lushness}\label{section:GL}

For  $x^*\in S_{X^*}$ and $\varepsilon > 0$ we put $S(x^*,\varepsilon): = \{x\in B_X:\; x^*(x) > 1-\varepsilon\}$. We say that a Banach space $X$ is \emph{lush}, if for every $x,y\in S_X$ and every $\varepsilon > 0$ there exists a functional $x^*\in S_{X^*}$ such that $x\in S(x^*,\varepsilon)$ and $\operatorname{dist}(y,\aco S(x^*,\varepsilon)) < \varepsilon$.

The concept of lushness, introduced by K. Boyko, V. Kadets, M. Mart{\'{\i}}n and D. Werner in \cite{BoKaMaWe07}, is a Banach space property, which ensures that the space has numerical index 1. It was used in \cite{BoKaMaWe07} to solve a problem concerning the numerical index of a Banach space. Lushness was further investigated e.g. in \cite{BoKaMaMe09} as a property of a Banach space. Later, D. Tan, X. Hunag and R. Liu in \cite{TanHuLiu13} proved that every lush space has ``Mazur-Ulam property'', that is, every isometry from a unit sphere of a lush space $E$ onto a unit sphere of a Banach space $F$ extends to a linear isometry of $E$ and $F$. Up to our knowledge, it is still an open problem whether every Banach space has Mazur-Ulam property. In order to prove that lush spaces have Mazur-Ulam property, the authors of \cite{TanHuLiu13} introduced the notion of generalized-lushness.

A Banach space $X$ is called \emph{generalized lush} (GL) if for every $x\in S_X$ and every $\varepsilon >0$ there is $x^*\in S_{X^*}$ such that $x\in S(x^*,\varepsilon)$ and, for every $y\in S_X$,
\[
\operatorname{dist}(y,S(x^*,\varepsilon)) + \operatorname{dist}(y,-S(x^*,\varepsilon)) < 2+\varepsilon.
\]

It is proved in \cite{TanHuLiu13} that every separable lush space is (GL) and that every (GL) space has Mazur-Ulam property. Hence, every separable lush space has Mazur-Ulam property. Now, using separable reduction, it is proved that lush spaces have Mazur-Ulam property \cite{TanHuLiu13}.
The concept of (GL) Banach spaces was further investigated as a property of a Banach space by J.-D. Hardtke \cite{har}.

The main purpose of this section is to prove Theorem~\ref{t:redukceGL}, that is, ``in Asplund spaces, generalized lushness is separably determined''. The fact that ``lushness is separably determined property'' was in some sense proved in \cite[Theorem 4.2]{BoKaMaMe09}. However,  in order to further combine it with other separable determination results, we need to prove this result in the language of rich families. Hence, we prove a slightly stronger version of \cite[Theorem 4.2]{BoKaMaMe09}.

\begin{thm}\label{t:lush}Let $X$ be a Banach space. Then there exists a rich family $\F\subset\S(X)$ such that for every $V\in\F$ we have
	\[
		X\text{ is lush }\Longleftrightarrow V\text{ is lush}.
	\]
\end{thm}

Further, we apply our results to prove that every Asplund lush space is (GL).

\begin{corollary}\label{c:lushIsGL}Let $X$ be an Asplund lush space. Then $X$ is (GL).\end{corollary}

This follows already from some known results, see Remark~\ref{r:lushIsGLOriginalProof}; however, we give a completely different proof of this fact.

\begin{proof}[Proof of Corollary~\ref{c:lushIsGL}]
Let $X$ be an Asplund lush space. By Theorem~\ref{t:lush}, Theorem~\ref{t:redukceGL} and Proposition~\ref{p:boMO} (we intersect rich families from Theorem~\ref{t:lush} and Theorem~\ref{t:redukceGL} and pick one space from the intersection), there is a closed separable subspace $V\subset X$ such that
\[
X\text{ is lush }\Longleftrightarrow V\text{ is lush}\qquad\quad\text{and}\qquad\quad V\text{ is (GL) }\implies X\text{ is (GL)}.
\]
Since $X$ is lush, $V$ is separable and lush; hence, by \cite[Example 2.5]{TanHuLiu13}, $V$ is (GL). Thus, $X$ is (GL).
\end{proof}
 
\begin{remark}\label{r:lushIsGLOriginalProof}Corollary~\ref{c:lushIsGL} follows already from some known results. Namely, by \cite[Theorem 2.1 (d)$\implies$ (c)]{BoKaMaMe09}, every lush Asplund space possesses the following property: for every $x\in S_X$ and every $\varepsilon > 0$, there is $x^*\in S_{X^*}$ such that $x\in S(x^*,\varepsilon)$ and for every $y\in S_X$ we have $\dist(y,\aco(S(x^*,\varepsilon))) < \varepsilon$. From this property it follows that $X$ is (GL).
\end{remark}
 
Let us start with the proof of Theorem~\ref{t:lush}. Similarly as in \cite{BoKaMaMe09}, we use the following result.

\begin{thm}[{\cite[Theorem 4.1]{BoKaMaMe09}}]\label{t:equivLush}Let $X$ be a real Banach space and $D\subset X$ a dense subspace. Then the following conditions are equivalent.
\begin{enumerate}[\upshape (i)]
	\item\label{lush:i} $X$ is lush
	\item\label{lush:ii} For every $x,y\in S_X$ and $\varepsilon > 0$, there are $\lambda_1,\lambda_2\geq 0$ with $\lambda_1 + \lambda_2 = 1$ and $x_1, x_2\in B_X$ such that $\|x + x_1 + x_2\| > 3 - \varepsilon$ and $\|y - (\lambda_1x_1 - \lambda_2 x_2)\| < \varepsilon$.
	\item\label{lush:iii} For every $x,y\in S_X\cap D$ and $\varepsilon > 0$, there are $\lambda_1,\lambda_2\geq 0$ with $\lambda_1 + \lambda_2 = 1$ and $x_1, x_2\in B_X$ such that $\|x + x_1 + x_2\| > 3 - \varepsilon$ and $\|y - (\lambda_1x_1 - \lambda_2 x_2)\| < \varepsilon$.
 \end{enumerate}
\end{thm}
\begin{proof}\eqref{lush:i}$\Leftrightarrow$\eqref{lush:ii} is proved in \cite[Theorem 4.1]{BoKaMaMe09}. The equivalence \eqref{lush:ii}$\Leftrightarrow$\eqref{lush:iii} is evident.
\end{proof}

\begin{proof}[Proof of Theorem~\ref{t:lush}]First, we will find a rich  family $\R_1\subset\S(X)$ such that for every $V\in\R_1$ we have
\[
X\text{ is lush }\implies V\text{ is lush}.
\]
If $X$ is not lush, we put $\R_1:=S(X)$. Otherwise, define $\R_1\subset\S(X)$ as the family consisting of all $V\in\S(X)$ such that $V$ is lush. We shall show that $\R_1$ is a rich family.
This is obvious if $X$ is not lush; hence, let us assume that $X$ is lush. By \cite[Theorem 4.2]{BoKaMaMe09}, $\R_1$ is cofinal. For checking the $\sigma$-completeness of $\R_1$, consider any increasing sequence $(V_i)_{i\in\en}$ of elements in $\R_1$. We need to prove that $V:=\overline{\bigcup_{i=1}^\infty V_i}$ is lush. Since $(V_i)_{i\in\en}$ is increasing, by Theorem~\ref{t:equivLush} \eqref{lush:i}$\implies$\eqref{lush:ii}, condition \eqref{lush:iii} in Theorem~\ref{t:equivLush}  is satisfied with $D = \bigcup_{i=1}^\infty V_i$ and $X = V$. Hence, $V$ is lush.

Now, we will find a rich  family $\R_2\subset\S(X)$ such that for every $V\in\R_2$ we have
\[
X\text{ is not lush }\implies V\text{ is not lush}.
\]
If $X$ is lush, we put $\R_2:=S(X)$. Otherwise, by Theorem~\ref{t:equivLush} \eqref{lush:ii}$\implies$\eqref{lush:i}, there are $x,y\in S_X$ and $\varepsilon > 0$ such that for every $\lambda_1,\lambda_2\geq 0$ with $\lambda_1 + \lambda_2 = 1$ and $x_1, x_2\in B_X$ we have $\|x + x_1 + x_2\| \leq 3 - \varepsilon$ or $\|y - (\lambda_1x_1 - \lambda_2 x_2)\| \geq \varepsilon$. Hence, the family $\R_2:=\{V\in\S(X)\setsep x,y\in V\}$ is a rich family such that each member of the family is not lush.

Finally, it remains to put $\A:=\R_1\cap\R_2$. This is a rich family because, by the construction above, we have $\A = \R_1$ or $\A =\R_2$ depending on the ``lushness'' of $X$. It is obvious that for every $V\in\A$, $V$ is lush if and only if $X$ is lush.
\end{proof}

In the remainder of this paper we prove the main result of this section, Theorem~\ref{t:redukceGL}. By Theorem \ref{t:main2} and Proposition \ref{p:boMO}, it follows from the following two results.

\begin{proposition}\label{p:redukceGL}Let $X$ be a Banach space. Then there exists a family $\F\subset\S(X)$ which is large in the sense of $\omega$-monotone mappings such that
\[
\forall F\in\F:\quad X\text{ is (GL)}\implies F\text{ is (GL).}
\]
\end{proposition}

\begin{proposition}\label{p:redukceGLAsplund}Let $X$ be an Asplund space. Then there exists a rich family $\A\subset\S(X)$ such that for every $V\in\A$ we have
	$$V\text{ is (GL) }\implies X\text{ is (GL)}.$$
\end{proposition}

First, let us note that in the definition of (GL) spaces we may work only with a dense subset of $X$ (resp. $X^*$). This is the content of the following two Lemmas. Since the proofs are straightforward and easy, we omit them.

\begin{lemma}\label{l:isGL}Let $X$ be a Banach space and let $D\subset X$ be a dense subset of $X$. Let us assume that for every $x\in D$ and $q\in(0,\infty)\cap \qe$ there exists $x^*\in X^*$ with $\tfrac{x}{\|x\|}\in S\Big(x^*,\varepsilon\Big)$ such that, for every $y\in D$,
\[
\dist\Big(\tfrac{y}{\|y\|},S(x^*,q)\Big) + \operatorname{dist}\Big(\tfrac{y}{\|y\|},-S(x^*,q)\Big) < 2+q.
\]
Then $X$ is (GL).
\end{lemma}
%\begin{proof}Fix $x\in S_X$ and $\varepsilon > 0$. Now, pick $q\in\qe\cap(0,\varepsilon)$ and $x_0\in D$ with $\|\tfrac{x_0}{\|x_0\|} - x\| < q/4$. By the assumptions, find $x^*\in S_{X^*}$ corresponding to $x_0$ and $q/4$. Then we have
%\[
%x^*(x) = x^*(\tfrac{x_0}{\|x_0\|}) - x^*(\tfrac{x_0}{\|x_0\|} - x) > 1 - q/4 - \|\tfrac{x_0}{\|x_0\|} - x\| > 1 - q/2 > 1 - \varepsilon
%\]
%hence, $x\in S(x^*,\varepsilon)$. Further, fix $y\in S_X$ and pick $y_0\in D$ with $\|y-\tfrac{y_0}{\|y_0\|}\| < q/4$. By the choice of $x^*$, we have
%\[
%\dist\bigg(\tfrac{y_0}{\|y_0\|},S\Big(x^*, q/4 \Big)\bigg) + \operatorname{dist}\bigg(\tfrac{y_0}{\|y_0\|},-S\Big(x^*, q/4 \Big)\bigg) < 2+ q/4.
%\]
%Since $S\Big(x^*, \varepsilon \Big) \supset S\Big(x^*, q/4 \Big)$, we have
%\[\begin{split}
%\dist\bigg(y,S\Big(x^*, \varepsilon \Big)\bigg) & + \operatorname{dist}\bigg(y,-S\Big(x^*, \varepsilon \Big)\bigg) < \dist\bigg(y,S\Big(x^*, q/4 \Big)\bigg) + \operatorname{dist}\bigg(y,-S\Big(x^*, q/4 \Big)\bigg) \\
%& \leq \dist\bigg(\tfrac{y_0}{\|y_0\|},S\Big(x^*, q/4 \Big)\bigg) + \operatorname{dist}\bigg(\tfrac{y_0}{\|y_0\|},-S\Big(x^*, q/4 \Big)\bigg) + 2 \|y - \tfrac{y_0}{\|y_0}\|\\
%& < 2+ q/4 + 2q/4 < 2 + q < 2 + \varepsilon.
%\end{split}\]
%Hence, $X$ is (GL).
%\end{proof}

\begin{lemma}\label{l:notGL}Let $X$ be a Banach space and let $G\subset X^*$ be a dense subset of $X^*$. Let us assume that there are $x\in S_X$ and $\varepsilon > 0$ such that for every $x^*\in G$ with $x\in S\Big(\tfrac{x^*}{\|x^*\|},\varepsilon\Big)$ there exists $y\in S_X$ such that
\[
\operatorname{dist}\bigg(y,S\Big(\tfrac{x^*}{\|x^*\|},\varepsilon\Big)\bigg) + \operatorname{dist}\bigg(y,-S\Big(\tfrac{x^*}{\|x^*\|},\varepsilon\Big)\bigg) \geq 2+\varepsilon.
\]
Then $X$ is not (GL).
\end{lemma}
%\begin{proof}In order to get a contradiction, assume that $X$ is (GL). Let $x\in S_X$ and $\varepsilon > 0$ be as in the assumptions. Since $X$ is (GL), there is $x_0^*\in S_{X^*}$ such that $x\in S(x_0^*,\varepsilon/2)$ and, for every $y\in S_X$, we have $\operatorname{dist}(y,S(x_0^*,\varepsilon/2)) + \operatorname{dist}(y,-S(x_0^*,\varepsilon/2)) < 2+\varepsilon/2$. Pick $\eta > 0$ such that, for every $x^*\in B(x_0^*,\eta)$, we have $S(x_0^*,\varepsilon/2)\subset S\Big(\tfrac{x^*}{\|x^*\|},\varepsilon\Big)$. Now, pick any $x^*\in B(x_0^*,\eta)\cap G$. Then we have $x\in S\Big(\tfrac{x^*}{\|x^*\|},\varepsilon\Big)$ and, for every $y\in S_X$,
%\[\begin{split}
%\operatorname{dist}\bigg(y,S\Big(\tfrac{x^*}{\|x^*\|},\varepsilon\Big)\bigg) & + \operatorname{dist}\bigg(y,-S\Big(\tfrac{x^*}{\|x^*\|},\varepsilon\Big)\bigg)\\
%&  \leq \operatorname{dist}\bigg(y,S\Big(x_0^*,\tfrac{\varepsilon}{2}\Big)\bigg) + \operatorname{dist}\bigg(y,-S\Big(x_0^*,\tfrac{\varepsilon}{2}\Big)\bigg) < 2+\varepsilon.
%\end{split}\]
%This contradicts the property of $x$ and $\varepsilon$ from the assumptions.
%\end{proof}

\begin{proof}[Proof of Proposition~\ref{p:redukceGL}]
If $X$ is not (GL), we may put $\F = \S(X)$. Let us assume that $X$ is (GL). For every $x\in X$ and $\varepsilon > 0$, we pick a point $I_1(x,\varepsilon)\in S_{X^*}$ such that $\tfrac{x}{\|x\|}\in S\big(I_1(x,\varepsilon),\varepsilon\big)$ and, for every $y\in X$,
$$\operatorname{dist}\Big(\tfrac{y}{\|y\|},S\big(I_1(x,\varepsilon),\varepsilon\big)\Big) + \operatorname{dist}\Big(\tfrac{y}{\|y\|},-S\big(I_1(x,\varepsilon),\varepsilon\big)\Big) < 2+\varepsilon.$$
Now, for every $x, y\in X$ and $\varepsilon > 0$, we pick two points $I_2(x,y,\varepsilon), I_3(x,y,\varepsilon)\in S\big(I_1(x,\varepsilon),\varepsilon\big)$ with
$$\left\|\tfrac{y}{\|y\|} - I_2(x,y,\varepsilon)\right\| + \left\|\tfrac{y}{\|y\|} + I_3(x,y,\varepsilon)\right\| < 2+\varepsilon.$$
Note that, since $X$ is (GL), for every $x,y\in X$ and $\varepsilon > 0$ the points $I_2(x,y,\varepsilon)$ and $I_3(x,y,\varepsilon)$ exist.

Now, let us define a mapping $\phi:[X]^{\leq\omega}\to[X]^{\leq\omega}$ by putting for every $C\in[X]^{\leq\omega}$
\[\begin{split}
	\phi_1(C) & : = \qespan C,\\
	\phi_{n+1}(C) & := \qespan\Big(\phi_n(C)\cup \big\{I_2(x,y,\varepsilon),I_3(x,y,\varepsilon)\setsep x,y\in \phi_n(C),\,\varepsilon\in\qe_+\big\}\Big),\qquad n\in\en,\\
	\phi(C) & : = \bigcup_{n\in\en}\phi_n(C).
\end{split}\]
We will show that $\F:=\{\overline{\phi(C)}\setsep C\in[X]^{\leq\omega}\}$ is the family we need. It is obvious that $\phi$ is monotone and $\phi(C)\supset C$ for every $C\in[X]^{\leq\omega}$. In order to prove $\omega$-monotonicity, pick an increasing sequence $(C_k)_{k\in\en}$ of countable subsets of $X$. By monotonicity, we have $\phi(\bigcup_{k=1}^\infty C_k)\supset \bigcup_{k=1}^\infty \phi(C_k)$. By induction on $n\in\en$, it is straightforward to verify that $\phi_n(\bigcup_{k\in\en} C_k)\subset \bigcup_{k=1}^\infty \phi(C_k)$; hence, we have $\phi(\bigcup_{k=1}^\infty C_k)\subset \bigcup_{k=1}^\infty \phi(C_k)$ and $\phi$ is $\omega$-monotone. Fix $C\in[X]^{\leq\omega}$. It remains to show that $\overline{\phi(C)}$ is (GL). Note that $\phi(C)$ is $\qe$-linear. It follows from the definition of $\phi(C)$ that we have
\begin{equation}
\label{eq:condRichIsGL}\forall x,y\in \phi(C)\; \forall\varepsilon\in\qe_+:\quad \{I_2(x,y,\varepsilon),I_3(x,y,\varepsilon)\}\subset \overline{\phi(C)}.
\end{equation}
We will verify the assumption of Lemma \ref{l:isGL} with $D = \phi(C)$ for the space $\overline{\phi(C)}$. Fix $x\in \phi(C)$ and $\varepsilon\in\qe_+$ and consider $x^*:=\tfrac{I_1(x,\varepsilon)|_{\overline{\phi(C)}}}{\|I_1(x,\varepsilon)|_{\overline{\phi(C)}}\|}$. Then $\tfrac{x}{\|x\|}\in S\big(I_1(x,\varepsilon),\varepsilon\big)\cap \overline{\phi(C)}\subset S(x^*,\varepsilon)$. Fix any $y\in \phi(C)$. Then we have $\{I_2(x,y,\varepsilon), I_3(x,y,\varepsilon)\}\subset S\big(I_1(x,\varepsilon),\varepsilon\big)\cap\overline{\phi(C)}\subset S(x^*,\varepsilon)$ and
\[
\left\|\tfrac{y}{\|y\|} - I_2(x,y,\varepsilon)\right\| + \left\|\tfrac{y}{\|y\|} + I_3(x,y,\varepsilon)\right\| < 2+\varepsilon.
\]
Hence,
\[
\operatorname{dist}\Big(\tfrac{y}{\|y\|},S\big(x^*,\varepsilon\big)\Big) + \operatorname{dist}\Big(\tfrac{y}{\|y\|},-S\big(x^*,\varepsilon\big)\Big) < 2+\varepsilon, 
\]
which shows that the assumption of Lemma \ref{l:isGL} is satisfied for the space $\overline{\phi(C)}$ and so $\overline{\phi(C)}$ is (GL).
\end{proof}

In order to prove the other implication (i.e. Proposition~\ref{p:redukceGLAsplund}), we restrict our attention to Asplund spaces. We recall the concept introduced in \cite{cuthFab} which serves as a link between $X$ and $X^*$ (and, by \cite[Theorem 2.3]{cuthFab}, exists right if and only if $X$ is Asplund).

\begin{definition}\label{generator}
\rm By an \emph{Asplund generator} in a Banach space $X$ we understand any correspondence 
$G:[X]^{\,\le\omega}\longrightarrow [X^*]^{\,\le\omega}$ such that
\smallskip

\noindent (a) $\big(\overline{\rm sp}\,C\big)^* = \overline {G(C)|_{\overline{\rm sp}\,C}}\ $
for every $C\in [X]^{\,\le\omega}$;

\smallskip
\noindent (b) if $C_1,\ C_2,\ \ldots$ is an increasing sequence in $[X]^{\,\le\omega}$, then
$G(C_1\cup C_2\cup\cdots)=G(C_1)\cup G(C_2)\cup\cdots\,$; 

\smallskip
\noindent (c) $\bigcup\{G(C):\ C\in [X]^{\,\le\omega}\}$ is a dense subset in $X^*$; and\smallskip

\noindent (d) if $C_1, C_2\in [X]^{\,\le\omega}$ are such that $\overline{\rm sp}\, C_1=\overline{\rm sp}\, C_2$,
then $\overline{\rm sp}\, G(C_1)=\overline{\rm sp}\, G(C_2)$.

\end{definition}

By $\S_{\oo}(X\times X^*)$ we denote the set $\{V\times Y\setsep V\in\S(X),\, Y\in\S(X^*)\}$. We say that $\R\subset \S_{\oo}(X\times X^*)$ is \emph{rich} if every member of $\S_{\oo}(X\times X^*)$ is contained in some $V\times Y\in\R$ and whenever we have an increasing sequence $(V_i\times Y_i)_{i\in\en}$ in $\R$, then $\overline{\bigcup_{i\in\en} V_i\times Y_i} = \overline{\bigcup_{i\in\en} V_i}\times\overline{\bigcup_{i\in\en} Y_i}  \in\R$.

\begin{proof}[Proof of Proposition \ref{p:redukceGLAsplund}]If $X$ is (GL), it suffices to put $\A:=\S(X)$. Therefore, we may assume that $X$ is not (GL).
Let $G:[X]^{\leq\omega}\to [X^*]^{\leq\omega}$ be an Asplund generator in $X$. Since $X$ is not (GL), there are $x_0\in S_X$ and $\varepsilon_0 > 0$ such that
	\begin{equation}
		\label{eq:notGLWeHave}\forall x^*\in S_{X^*}: x_0 \in S(x^*,\varepsilon_0)\quad \exists y\in S_X: \operatorname{dist}(y,S(x^*,\varepsilon_0)) + \operatorname{dist}(y,-S(x^*,\varepsilon_0)) \geq 2+\varepsilon_0.
	\end{equation}
	By Lemma \ref{l:notGL} and the definition of an Asplund generator, it suffices to find a rich family $\A\subset\S(X)$ such that for every $V\in\A$ we have $x_0\in V$ and there exists $C\subset V$ with $\closedSpan C = V$ satisfying the following property.
	\begin{equation}\begin{split}
		\label{eq:notGLWeWant}\forall x^*\in G(C): \;& x_0 \in S\Big(\tfrac{x^*|_V}{\|x^*|_V\|},\varepsilon_0\Big)\quad \exists y\in S_V\\
		& \operatorname{dist}\bigg(y,S\Big(\tfrac{x^*|_V}{\|x^*|_V\|},\varepsilon_0\Big)\bigg) + \operatorname{dist}\bigg(y,-S\Big(\tfrac{x^*|_V}{\|x^*|_V\|},\varepsilon_0\Big)\bigg) \geq 2+\varepsilon_0.
	\end{split}\end{equation}
	
	Define $\R'\subset\S_{\oo}(X\times X^*)$ as the family consisting of all rectangles $\overline{\rm sp}\, C\times \overline{\rm sp}\, G(C)$, with $C\in[X]^\en$, such that the assignment
	\[
\overline{\rm sp}\, G(C)\ni x^*\longmapsto x^*{}|\,_{\overline{\rm sp}\, C}\in(\overline{\rm sp}\, C)^* 
\]
is a surjective isometry. It is proved in \cite[proof of Theorem 2.3 (ii)$\implies$(iii)]{cuthFab} that $\R'$ is a rich family and whenever we have $V_1\times Y_1$, $V_2\times Y_2$ in $\R'$ such that $V_1\subset V_2$, then $Y_1\subset Y_2$. Consequently, the family $\R_1 := \{V\setsep \exists Y:\,V\times Y\in\R'\}\subset \S(X)$ is rich.
	
	For every $x^*\in X^*$, we pick, if it exists, a point $I(x^*)\in S_X$ such that
	\begin{equation}
		\label{eq:notGLCondition} \operatorname{dist}\bigg(I(x^*),S\Big(\tfrac{x^*}{\|x^*\|},\varepsilon_0\Big)\bigg) + \operatorname{dist}\bigg(I(x^*),-S\Big(\tfrac{x^*}{\|x^*\|},\varepsilon_0\Big)\bigg) \geq 2 + \varepsilon_0.
	\end{equation}
	Define $\R_2\subset\S(X)$ as the family consisting of all $V\in\S(X)$ with $x_0\in V$ such that there is a countable set $C\subset V$ with $\closedSpan C = V$ and
	\begin{equation}
		\label{eq:richConditionNotGLDense} \forall x^*\in G(C):\quad \big(I(x^*)\text{ is defined }\implies I(x^*)\in V\big).
	\end{equation}	
	We shall show that $\R_2$ is a rich family.		
	
	As regards the cofinality of $\R_2$, fix any countable set $S\subset X$. Put $C_0 := S\cup\{x_0\}$. Assume that for some $m\in \en$ we already found countable sets $C_0 \subset C_1 \subset\ldots\subset C_{m-1} \subset X$. Then we find $C_m\supset C_{m-1}$ such that, for every $x^*\in G(C_{m-1})$, we have $I(x^*)\in C_m$ whenever it is defined. Do so for every $m\in \en$ and put finally
$C := \bigcup_{i=0}^\infty C_i$. It remains to see that $V:=\overline{\spn}\; C\in\R_2$, which follows immediately from the construction because we have $G(C) = \bigcup_{i=0}^\infty G(C_i)$.

For checking the $\sigma$-completeness of $\R_2$, consider any increasing sequence $(V_i)_{i\in\en}$ of elements in $\R_2$. Let, for every $i\in\en$, be $C_i\subset V_i$ a set with $\closedSpan C_i = V_i$ satisfying \eqref{eq:richConditionNotGLDense} for $C_i$ and $V_i$. We may assume that $C_1\subset C_2\subset\ldots$ (if  not, we replace it by $C_1, C_1\cup C_2, C_1\cup C_2\cup C_3, \ldots$). Then $V = \overline{V_1\cup V_2\cup\ldots}$ contains $x_0$ and we put $C:=C_1\cup C_2\ldots$. Then $\closedSpan C = V$. Moreover, since $C_1,C_2,\ldots$ is an increasing sequence, \eqref{eq:richConditionNotGLDense} is satisfied.

Finally, we put $\A : = \R_1\cap \R_2$. It remains to prove that our $\A$ ``works''; i.e., no member of $\A$ is (GL). So, pick any $V\in\A$. We need to show that \eqref{eq:notGLWeWant} holds. Fix a set $C$ with $\closedSpan C = V$ from the definition of the family $\R_2$. Fix $x^*\in G(C)$ with $x_0\in S\Big(\tfrac{x^*|_{V}}{\|x^*|_{V}\|},\varepsilon_0\Big)$. By the definiton of $\R_1$, there is a countable set $C'\subset V$ such that $\closedSpan C' = V$ and, for every $y^*\in \closedSpan G(C')$, we have $\|y^*\| = \|y^*|_{V}\|$. By the definition of an Asplund generator, $\closedSpan G(C) = \closedSpan G(C')$; thus, we have $\|x^*\| = \|x^*|_{V}\|$. Hence, we have $x_0\in S(x^*,\varepsilon_0)$ and, by \eqref{eq:notGLWeHave}, $I(x^*)$ is defined; hence, by \eqref{eq:richConditionNotGLDense}, we have $y:=I(x^*)\in S_V$. Consequently,
	\[\begin{split}
		\operatorname{dist}\bigg(y,\, & S\Big(\tfrac{x^*|_V}{\|x^*|_V\|},\varepsilon_0\Big)\bigg) + \operatorname{dist}\bigg(y,-S\Big(\tfrac{x^*|_V}{\|x^*|_V\|},\varepsilon_0\Big)\bigg)\\
		& = \operatorname{dist}\bigg(y,S\Big(\tfrac{x^*}{\|x^*\|},\varepsilon_0\Big)\cap B_V\bigg) + \operatorname{dist}\bigg(y,-S\Big(\tfrac{x^*}{\|x^*\|},\varepsilon_0\Big)\cap B_V\bigg)\\
		& \geq \operatorname{dist}\bigg(y,S\Big(\tfrac{x^*}{\|x^*\|},\varepsilon_0\Big)\bigg) + \operatorname{dist}\bigg(y,-S\Big(\tfrac{x^*}{\|x^*\|},\varepsilon_0\Big)\bigg) \stackrel{\eqref{eq:notGLCondition}}{\geq} 2 + \varepsilon_0.
	\end{split}\]
Thus, \eqref{eq:notGLWeWant} holds and $V$ is not (GL).
\end{proof}

\section*{Acknowledgements}

The author would like to thank L. Zaj\'i\v{c}ek and M. Fabian for  their valuable comments which improved the content of the introductory section and to an anonymous referee of Comptes Rendus Mathematique who observed the content of Remark \ref{r:lushIsGLOriginalProof}.


\begin{thebibliography}{10}

\bibitem{boMo}
{\sc J.~M. Borwein and W.~B. Moors}, {\em Separable determination of
  integrability and minimality of the {C}larke subdifferential mapping}, Proc.
  Amer. Math. Soc., 128 (2000), pp.~215--221.

\bibitem{BoKaMaMe09}
{\sc K.~Boyko, V.~Kadets, M.~Mart{\'{\i}}n, and J.~Mer{\'{\i}}}, {\em
  Properties of lush spaces and applications to {B}anach spaces with numerical
  index 1}, Studia Math., 190 (2009), pp.~117--133.

\bibitem{BoKaMaWe07}
{\sc K.~Boyko, V.~Kadets, M.~Mart{\'{\i}}n, and D.~Werner}, {\em Numerical
  index of {B}anach spaces and duality}, Math. Proc. Cambridge Philos. Soc.,
  142 (2007), pp.~93--102.

\bibitem{c}
{\sc M.~C{\'u}th}, {\em Separable reduction theorems by the method of
  elementary submodels}, Fund. Math., 219 (2012), pp.~191--222.

\bibitem{cuthFab}
{\sc M.~C{\'u}th and M.~Fabian}, {\em Asplund spaces characterized by rich
  families and separable reduction of {F}r\'echet subdifferentiability}, J.
  Funct. Anal., 270 (2016), pp.~1361--1378.

\bibitem{cuthKalenda}
{\sc M.~C{\'u}th and O.~F.~K. Kalenda}, {\em Rich families and elementary
  submodels}, Cent. Eur. J. Math., 12 (2014), pp.~1026--1039.

\bibitem{cr}
{\sc M.~C{\'u}th and M.~Rmoutil}, {\em {$\sigma$}-porosity is separably
  determined}, Czechoslovak Math. J., 63(138) (2013), pp.~219--234.

\bibitem{crz}
{\sc M.~C{\'u}th, M.~Rmoutil, and M.~Zelen{\'y}}, {\em On separable
  determination of {$\sigma$}-{$\mathbf{P}$}-porous sets in {B}anach spaces},
  Topology Appl., 180 (2015), pp.~64--84.

\bibitem{dow}
{\sc A.~Dow}, {\em An introduction to applications of elementary submodels to
  topology}, Topology Proc., 13 (1988), pp.~17--72.

\bibitem{FaIo}
{\sc M.~Fabian and A.~Ioffe}, {\em Separable {R}eductions and {R}ich {F}amilies
  in the {T}heory of {F}r\'echet {S}ubdifferentials}, J. Convex Anal., 23
  (2016), pp.~631--648.

\bibitem{har}
{\sc J.-D. Hardtke}, {\em Some remarks on generalised lush spaces},  (2013).
\newblock preprint avaiable at \url{http://arxiv.org/pdf/1309.4358.pdf}.

\bibitem{kos}
{\sc P.~Koszmider}, {\em Projections in weakly compactly generated {B}anach
  spaces and {C}hang's conjecture}, J. Appl. Anal., 11 (2005), pp.~187--205.

\bibitem{kubis}
{\sc W.~Kubi{\'s}}, {\em Banach spaces with projectional skeletons}, J. Math.
  Anal. Appl., 350 (2009), pp.~758--776.

\bibitem{k}
{\sc K.~Kunen}, {\em Set theory}, vol.~102 of Studies in Logic and the
  Foundations of Mathematics, North-Holland Publishing Co., Amsterdam-New York,
  1980.
\newblock An introduction to independence proofs.

\bibitem{LiMo08}
{\sc P.~Lin and W.~B. Moors}, {\em Rich families, {$W$}-spaces and the product
  of {B}aire spaces}, Math. Balkanica (N.S.), 22 (2008), pp.~175--187.

\bibitem{LiPrTi}
{\sc J.~Lindenstrauss, D.~Preiss, and J.~Ti{\v{s}}er}, {\em Fr\'echet
  differentiability of {L}ipschitz functions and porous sets in {B}anach
  spaces}, vol.~179 of Annals of Mathematics Studies, Princeton University
  Press, Princeton, NJ, 2012.

\bibitem{MoSp}
{\sc W.~B. Moors and J.~Spurn{\'y}}, {\em On the topology of pointwise
  convergence on the boundaries of {$L_1$}-preduals}, Proc. Amer. Math. Soc.,
  137 (2009), pp.~1421--1429.

\bibitem{Pr84}
{\sc D.~Preiss}, {\em G\^ateaux differentiable functions are somewhere
  {F}r\'echet differentiable}, Rend. Circ. Mat. Palermo (2), 33 (1984),
  pp.~122--133.

\bibitem{rojTka}
{\sc R.~Rojas-Hern{\'a}ndez and V.~V. Tkachuk}, {\em A monotone version of the
  {S}okolov property and monotone retractability in function spaces}, J. Math.
  Anal. Appl., 412 (2014), pp.~125--137.

\bibitem{TanHuLiu13}
{\sc D.~Tan, X.~Huang, and R.~Liu}, {\em Generalized-lush spaces and the
  {M}azur-{U}lam property}, Studia Math., 219 (2013), pp.~139--153.

\bibitem{To81}
{\sc H.~Toru{\'n}czyk}, {\em Characterizing {H}ilbert space topology}, Fund.
  Math., 111 (1981), pp.~247--262.

\bibitem{vz}
{\sc L.~Vesel{\'y} and L.~Zaj{\'{\i}}{\v{c}}ek}, {\em On differentiability of
  convex operators}, J. Math. Anal. Appl., 402 (2013), pp.~12--22.

\bibitem{z}
{\sc L.~Zaj{\'{\i}}{\v{c}}ek}, {\em Generic {F}r\'echet differentiability on
  {A}splund spaces via a.e.\ strict differentiability on many lines}, J. Convex
  Anal., 19 (2012), pp.~23--48.

\end{thebibliography}
\end{document}